\documentclass[a4paper,12pt,intlimits,oneside]{amsart}

\usepackage{amsmath}
\usepackage{amssymb}
\usepackage{amsfonts}
\usepackage{amsthm,amscd}

\theoremstyle{definition}
\newtheorem{thm}{Theorem}[section]
\newtheorem{prop}[thm]{Proposition}

\newtheorem{lem}[thm]{Lemma}
\newtheorem{defn}[thm]{Definition}

\def\E{{\mathfrak E}}
\def\k{{\kappa}}
\newcommand{\comment}[1]{}

\numberwithin{equation}{section}

\def\lsim{\raisebox{-1ex}{$~\stackrel{\textstyle <}{\sim}~$}}

\theoremstyle{definition}

\begin{document}
\title{Maximal operator in Dunkl-Fofana spaces}
\author[P. Nagacy]{Pokou Nagacy}
\address{Laboratoire de Math\'ematiques Fondamentales, UFR Math\'ematiques et Informatique, Universit\'e F\'elix Houphou\"et-Boigny Abidjan-Cocody, 22 B.P 582 Abidjan 22. C\^ote d'Ivoire}
\email{{\tt pokounagacy@yahoo.com}}
\author[J. Feuto]{Justin Feuto}
\address{Laboratoire de Math\'ematiques Fondamentales, UFR Math\'ematiques et Informatique, Universit\'e F\'elix Houphou\"et-Boigny Abidjan-Cocody, 22 B.P 1194 Abidjan 22. C\^ote d'Ivoire}
\email{{\tt justfeuto@yahoo.fr}}


\date{}

\begin{abstract}
 We generalize Wiener amalgam spaces by using Dunkl translation instead of the classical one, and we give some relationship between these spaces, Dunkl-Lebesgue spaces and Dunkl-Morrey spaces. We prove that the Hardy-Litlewood maximal function associated with the Dunkl operators is bounded on these generalized Dunkl-Morrey spaces. 
\end{abstract}

\maketitle

\section{Introduction}
For $1\leq p,q\leq\infty$, the amalgam of $L^q$ and $L^p$ on the real line is the space $(L^{q},L^{p})$ of complex values  functions $f$ on  $\mathbb R$ which are locally in $L^{q}$ and such that the function $y\mapsto \left\|f\chi_{I(y,1)}\right\|_{L^q}$ belongs to $L^p(\mathbb R)$, where $\chi_{I(y,1)}$ denotes the characteristic function of the open interval $I(y,1)=(y-1,y+1)$ and $ \left\|\cdot\right\|_{L^q}$ the usual Lebesgue norm in $L^q$.

We recall that for $q<\infty$
$$\left\|f\chi_{I(y,1)}\right\|_{L^q}=\left(\int_{I(0,1)}\left|f(u+y)\right|^{q}du\right)^{\frac{1}{q}}=\left\|\;(_{-y}\left|f\right|^q)\chi_{I(0,1)}\right\|_{L^q},$$

where $(_{y}g)(x)=g(x-y)$ is the classical translation.

Over the past 30 years, much work in harmonic analysis has focused on generalizing the results established in classical Fourier analysis within the framework of Dunkl analysis.
The Dunkl translation (see Section 2) is an important tool in this task. It is natural to seek to know what become Wiener amalgam spaces once we replace classical translation by Dunkl one.

Taking in the definition of Wiener amalgam space the generalized translation as defined in \cite{T5}, we give the analogue of the Wiener amalgam spaces and look at some subspaces of this spaces  as well as the boundedness of the Hardy-Littlewood maximal operator associated with the Dunkl operator. For this purpose we define the function spaces using the harmonic analysis associated with the Dunkl operators on $\mathbb R$. The generalized
shift operators we are considering are associated with the reflection group $\mathbb Z_{2}$ on $\mathbb R$. For the basic properties of the Dunkl analysis, we refer the reader to \cite{D1,R4,T5} and the references therein. The Lebesgue measure on the real line will be denoted by $dx$. 
For any Lebesgue measurable subset $E$ of $\mathbb R$, $\left|E\right|$ stands for its Lebesgue measure. We denote by $\mathcal E$ the space of functions of class $\mathcal C^\infty$ and by $\mathcal C_c$ the space of continuous functions with compact supports.

This paper is organized as follow:

 Section 2 is devoted to  the prerequisite for Dunkl analysis on the real line while Section 3 deals with the definition and properties of Dunkl-amalgam spaces. We introduce Fofana spaces associated to the Dunkl-amalgam in Section 4, and compare this space with Lebesgue spaces and Morrey-Dunkl spaces. In Section 5, we stated our results for Dunkl-type Hardy-Littlewood maximal function and give their proofs in Section 6. 

The letter $C$ will be used for non-negative constants independent of the relevant variables, and this constant may change from one occurrence to another. A constant with index, such as $ C_\kappa $, does not change in different occurrences but depends on the parameters mentioned there.

We propose the following abbreviation $\mathrm{\bf A}\lsim \mathrm{\bf B}$ for the inequalities $\mathrm{\bf A}\leq C\mathrm{\bf B}$, where $C$ is a positive constant independent of the main parameters. If $\mathrm{\bf A}\lsim \mathrm{\bf B}$ and $\mathrm{\bf B}\lsim \mathrm{\bf A}$, then we write $\mathrm{\bf A}\approx \mathrm{\bf B}$.

\section{Prerequisite on Dunkl analysis on the real line}
For a real parameter $\kappa >-1/2$, we consider the Dunkl operator, associated
with the reflection group $\mathbb Z_{2}$ on $\mathbb R$:
$$\Lambda_{\kappa}f(x)=\frac{df}{dx}(x)+\frac{2\kappa+1}{x}\left(\frac{f(x)-f(-x)}{2}\right).$$
For $\lambda\in\mathbb C$, the Dunkl kernel denoted $\E_{\kappa}(\lambda)$ (see \cite{D2}), is the only solution of the initial value problem 
$$\Lambda_{\kappa}f(x)=\lambda f(x),\ f(0)=1,\ x\in\mathbb R.$$
It is given by the formula
\begin{equation*}
\E_{\kappa}(\lambda x)=j_{\kappa}(i\lambda x)+\frac{\lambda x}{2(\kappa+1)}j_{\kappa+1}(i\lambda x),\ x\in\mathbb R,
\end{equation*}
where 
$$j_\kappa(z)=2^\kappa\Gamma(\kappa+1)\frac{J_\kappa(z)}{z^\kappa}=\Gamma(\kappa+1)\sum^{\infty}_{n=0}\frac{(-1)^nz^{2n}}{n!2^{2n}\Gamma(n+\kappa+1)},\quad z\in \mathbb C$$
is the normalized Bessel function of the first kind and of order $\kappa$. 
Notice that $\Lambda_{-\frac{1}{2}}=\frac{d}{dx}$ and $\E_{-\frac{1}{2}}(\lambda x)=e^{\lambda x}$.
 It is also proved (see \cite{R4}) that $\left|\E_{\kappa}(ix)\right|\leq 1$ for every $x\in\mathbb R$.

In the rest of the paper, we fix $\kappa > -1/2$, and consider the weighted Lebesgue measure $\mu$ on $\mathbb R$, given by
\begin{equation}
d\mu(x):=\left(2^{\kappa+1}\Gamma(\kappa+1)\right)^{-1}\left|x\right|^{2\kappa+1}dx.\label{*}
\end{equation}
For $1\leq p\leq\infty$, we denote by $ L^{p}(\mu)$ the Lebesgue space associated with the measure $\mu$, while $L^{0}(\mu)$ stands for the complex vector space of equivalence classes (modulo equality $\mu$-almost everywhere) of complex-valued functions $\mu$-measurable  on $\mathbb R$. For $f\in L^{p}(\mu)$, we denote by $\Vert f\Vert_p$ the classical norm of $f$.

The Dunkl kernel $\E=\E_{\kappa}$ gives rise to an integral transform on $\mathbb R$ denoted $\mathcal F=\mathcal F_{\kappa}$ and called Dunkl transform (see \cite{MFE}). For $f\in L^{1}(\mu)$ 
 $$\mathcal F f(\lambda)=\int_{\mathbb R}\mathfrak E(-i\lambda x)f(x)d\mu(x),\ \ \lambda\in\mathbb R.$$
We have the following properties of the Dunkl transform  given in  \cite{MFE}  (see also \cite{D1}).
\begin{thm}
\begin{enumerate}
\item Let $f\in L^1(\mu)$. If $\mathcal F (f)$ is in $L^1(\mu)$, then we have the following inversion formula :
$$f(x)=c\int_{\mathbb R}\E(ixy)\mathcal F(f)(y)d\mu(y).$$
\item The Dunkl transform has a unique extension to an isometric isomorphism on $L^2(\mu)$.
\end{enumerate}
\end{thm}
\begin{defn}
Let $y\in\mathbb R$ be given. The generalized translation operator $f\mapsto\tau_y f$ is defined on $L^2(\mu)$ by the equation 
$$\mathcal F(\tau_yf)(x)=\E(ixy)\mathcal F(f)(x);\quad x\in\mathbb R$$
\end{defn}
Mourou proved in \cite{M} that generalized translation operators have the following properties:
\begin{thm}
\begin{enumerate}
\item The operator $\tau_{x}$, $x\in\mathbb R$, is a continuous linear operator from $\mathcal E(\mathbb R)$ into itself,
\item For $f\in\mathcal E(\mathbb R)$ and  $x,y\in \mathbb R$ 
\begin{enumerate}
\item $\tau_{x}f(y)=\tau_{y}f(x)$,
\item $\tau_{0}f=f$,
\item $\tau_{x}\circ\tau_{y}=\tau_{y}\circ\tau_{x}$.
\end{enumerate}
\end{enumerate}
\end{thm}
For $x\in\mathbb R$, let $B(x,r)=\left\{y\in\mathbb R:\max\left\{0,\left|x\right|-r\right\}<\left|y\right|<\left|x\right|+r\right\}$ if $x\neq 0$ and $B_r:=B(0,r)=\left]-r,r\right[$. We  have $\mu(B_r)=\mathfrak b_\kappa r^{2\kappa+2}$ where $\mathfrak b_\kappa=\left[2^{\kappa+1}(\kappa+1)\Gamma(\kappa+1)\right]^{-1}$, and for $f\in L^1_{loc}(\mu)$, the following analogue of the Lebesgue differentiation theorem (see  \cite{S}).

\begin{equation}
\lim_{r\rightarrow 0}\frac{1}{\mu(B_r)}\int_{B_r}\left|\tau_{x}f(y)-f(x)\right|d\mu(y)=0\ \ \text{ for a. e. }x\in\mathbb R\label{translate1}
\end{equation}
and
\begin{equation}
\lim_{r\rightarrow 0}\frac{1}{\mu(B_r)}\int_{B_r}\tau_{x}f(y)d\mu(y)=f(x)\ \ \text{ for a. e. }x\in\mathbb R.\label{translate2}
\end{equation}

 For $x\in\mathbb R $ and $r>0$, the map $y\mapsto\tau_{x}\chi_{B_r}(y)$  is supported in $B(x,r)$ and 
 \begin{equation}
 0\leq \tau_{x}\chi_{B_r}(y)\leq\min\left\{1,\frac{2C_\kappa}{2\kappa+1}\left(\frac{r}{\left|x\right|}\right)^{2\kappa+1}\right\},\  y \in B(x,r),\label{**}
 \end{equation}
 as proved in \cite{GM3}.
 
Let $f$ and $g$ be two continuous functions on $\mathbb R$ with compact support. We
define the generalized convolution $\ast_{\kappa}$ of $f$ and $g$ by
\begin{equation*}
f\ast_{\kappa}g(x)=\int_{\mathbb R}\tau_{x}f(-y)g(y)d\mu(y).
\end{equation*}
The generalized convolution $\ast_{\kappa}$ is associative and commutative \cite{R4}. We also have the following results.
\begin{prop}[see Soltani \cite{S}]\label{young}
\begin{enumerate}
\item For all $x\in\mathbb R$, the operator $\tau_{x}$ extends to $L^{p}(\mu)$, $p\geq 1$, and 
\begin{equation}
\left\|\tau_{x}f\right\|_{p}\leq 4\left\|f\right\|_{p}\label{conttranslation}
\end{equation}
for all $f\in L^{p}(\mu)$.
\item Assume that $p, q, r \in\left[1,\infty\right]$ and satisfy $\frac{1}{p} +\frac{1}{q} = 1+\frac{1}{r}$. Then the generalized convolution defined on
$\mathcal C_{c}\times\mathcal C_{c}$, extends to a continuous map from $L^{p}(\mu)\times L^{q}(\mu)$ to $L^{r}(\mu)$, and we have
\begin{equation*}
\left\|f\ast_{\kappa}g\right\|_{r}\leq 4\left\|f\right\|_{p}\left\|g\right\|_{q}.
\end{equation*}
\end{enumerate}
\end{prop}
It is also proved in \cite{GM2} that if $f\in L^1(\mu)$ and $g\in L^p(\mu)$, $1\leq p<\infty$, then 
\begin{equation*}
\tau_x(f\ast g)=\tau_x f\ast g=f\ast \tau_x g,x\in\mathbb{R}.
\end{equation*}

\section{Dunkl-Wiener amalgam spaces}

For $1\leq q,p\leq\infty$, the generalized amalgam space $(L^q,L^p)(\mu)$ is defined by 
 $$(L^{q},L^{p})(\mu)=\left\{f\in L^{0}(\mu):\ \left\|f\right\|_{q,p}<\infty\right\},$$
  where
  \begin{equation*}
 \left\|f\right\|_{q,p}=\left\{\begin{array}{lll}\left\|\left[\int_{\mathbb R}(\tau_{y}\left|f\right|^{q})\chi_{B_1}(x)d\mu(x)\right]^{\frac{1}{q}}\right\|_p&\text{ if }&q<\infty\\
 \left\|\left\|f\chi_{B(y,1)}\right\|_{\infty}\right\|_p&\text{ if }&q=\infty
 \end{array}\right.,
 \end{equation*}
 and the $L^p(\mu)$-norm is taken with respect to the $y$ variable. We recover the classical Wiener amalgam spaces by taking the Lebesgue measure instead of $\mu$, and the classical translation instead of that of Dunkl.

It is easy to see that for $1\leq q\leq p\leq\infty$, $(L^{q},L^{p})(\mu)$ is a complex subspace of the space $L^{0}(\mu)$. Notice also that $(L^\infty,L^\infty)(\mu)=L^\infty(\mu)$. In fact, since 
$$\left|f\chi_B(y,1)\right|\leq \left|f\right|, y\in\mathbb R,$$
we have $\left\| f \right\|_{+\infty,+\infty} \leq \left\|f\right\|_{+\infty}.$

Conversely, we assume that $\left\|f\right\|_{\infty}>0$ (since otherwise there is nothing to prove). Fix  $0<r< \left\|f\right\|_{\infty}$. We have 
	$$\mu(\left\{x \in \mathbb{R}/  \left|f(x)\right|  > r	\right\} ) > 0.$$
	Thus, there is a compact set $K\subset\left\{x \in \mathbb{R}/  \left|f(x)\right|  > r	\right\}$ satisfying $\mu(K)>0$. Since $K \subset \cup^{n}_{i=1} B(y_{i},\frac{1}{2})$ for a finite sequence $y_{1},y_{2},\ldots,y_{n} \in K$, there exists  $1\leq i_{0}\leq n$ such that  $\mu( K \cap B(y_{i_{0} },\frac{1}{2}))> 0$. It is easy to see that 
	$$B(y_{i_0},\frac{1}{2})\subset B(y,1),\text{ for all } y\in B(y_{i_0},\frac{1}{2}).$$
	It comes that  
	$$B(y_{i_0},\frac{1}{2})\subset\left\{y\in\mathbb R/\left\|f\chi_{B(y,1)}\right\|_{\infty}>r\right\},$$
	so that 
	$$r<\left\|f\right\|_{\infty,\infty}.$$
This implies that $\left\|f\right\|_{\infty}\leq\left\|f\right\|_{\infty,\infty}$. 

H\"older's inequality is also valid in these spaces. 
\begin{prop}
	Let $1\leq q_{1} , p_{1}\leq\infty$ and $1\leq q_{2} , p_{2}\leq \infty$ such that
	$$\frac{1}{q_{1}} + \frac{1}{q_{2}} = \frac{1}{q} \leq 1\text{ and }\frac{1}{p_{1}} + \frac{1}{p_{2}} = \frac{1}{p} \leq 1.$$
	If $f\in(L^{q_{1}}, L^{p_{1}}) (\mu)$ and $g\in(L^{q_{2}}, L^{p_{2}})(\mu) $ then $fg\in(L^{q}, L^{p}) (\mu)$. Moreover, we have 
	$$\left\| fg \right\|_{q,p} \leq \left\| f \right\|_{q_{1},p_{1}}  \left\| g \right\|_{q_{2},p_{2}}.$$
\end{prop}

\begin{proof}
Let $f\in (L^{q_{1}}, L^{p_{1}})(\mu)$ and $g\in (L^{q_{2}}, L^{p_{2}})(\mu) $.

We suppose that $q_1<\infty$ and $q_2<\infty$. For almost every $y$ in $\mathbb R$, the maps
$x\mapsto \left|f(x)\right|(\tau_{y}\chi_{B(0,1)}(-x))^{\frac{1}{q_1}}$ and $x\mapsto \left|g(x)\right|(\tau_{y}\chi_{B(0,1)}(-x))^{\frac{1}{q_2}}$ are respectively in $L^{q_1}(\mu)$ and $L^{q_2}(\mu)$ so that, applying H\"older inequality in Lebesgue space, we have
\begin{eqnarray*}\left[\int_{\mathbb R}(\left|f(x)\right| \left|g(x)\right|)^q(\tau_{y}\chi_{B(0,1)}(-x)) d\mu(x)\right]^{\frac{1}{q}}&\leq&\left[ \int_{\mathbb R}\left|f(x)\right|^{q_1} \tau_{y}\chi_{B(0,1)}(-x)) d\mu(x)\right]^{\frac{1}{q_1}}\\
&\times&\left[\int_{\mathbb R} \left|g(x)\right|^{q_2}\tau_{y}\chi_{B(0,1)}(-x) d\mu(x)\right]^{\frac{1}{q_2}}.
\end{eqnarray*}
Taking the $L^p(\mu)$-norm of both sides and applying once more the H\"older inequality on the term in the right hand side, we obtain the desired result.

We suppose known that some $q_i=\infty$ let say $q_1$, and $q_2<\infty$. Then $q=q_2$ and the result follows from the following inequality
$$\begin{aligned}&\left[\int_{\mathbb R}(\left|f(x)\right| \left|g(x)\right|)^{q_2}(\tau_{y}\chi_{B(0,1)}(-x)) d\mu(x)\right]^{\frac{1}{q_2}}\\
&\qquad \qquad \qquad \qquad \qquad  \qquad   \leq \left\|f\chi_{B(y,1)}\right\|_\infty\left[\int_{\mathbb R} \left|g(x)\right|^{q_2}(\tau_{y}\chi_{B(0,1)}(-x)) d\mu(x)\right]^{\frac{1}{q_2}},
\end{aligned}$$
by taking the $L^p(\mu)$-norm of both sides and applying H\"older inequality for $p_i$ indexes. 

In the case $q_1=q_2=\infty$, the result is immediate.
\end{proof}
\begin{prop}\label{norm1} 
Let  $1\leq q\leq p\leq \infty $. The maps $ L^{0}(\mu)\ni f\mapsto \left\|f\right\|_{q,p}$ 
define a norm in $(L^{q},L^{p})(\mu)$.
\end{prop}
    
\begin{proof}
 Let $f$ be a $\mu$-measurable function such that  $\left\|f\right\|_{q,p}=0$.
 \begin{enumerate}
 \item Suppose $q<\infty$, then
 \begin{equation*}
 0=\left\|f\right\|_{q,p}=\left\|\left[\int_{\mathbb R}(\tau_{(\cdot)}\left|f\right|^{q})\chi_{B_1}(x)d\mu(x)\right]^{\frac{1}{q}}\right\|_{p}
 \end{equation*}
implies that there exists a $\mu$-null set $N$ such that for all $y\in\mathbb R\setminus N$, we have 
$$0=\int_{\mathbb R}\tau_y\left|f(x)\right|^q\chi_{B_1}(x)d\mu(x)=\int_{\mathbb R}\left|f(x)\right|^q\tau_{-y}\chi_{B_1}(x)d\mu(x). $$
Thus for $y\in\mathbb R\setminus N$, there exists  a subset $N_{y}\subset\mathbb R$ with $\mu(N_{y})=0$, such that for $x\in\mathbb R\setminus N_{y}$ 
$$\left|f\right|^{q}\tau_{-y}\chi_{B_1}(x)=0.$$
Fix $y\in\mathbb R\setminus N$. For every $\eta\leq 1$ we have $B_\eta\subset B_1$, so that $\tau_{-y}(\chi_{B_1}-\chi_{B_\eta})\geq 0$. Hence 
\begin{eqnarray*}
\frac{1}{\mu(B_\eta)}\int_{B_\eta}(\tau_{y}\left|f\right|^{q})(x)d\mu(x)&=&\frac{1}{\mu(B_\eta)}\int_{\mathbb R}\left|f\right|^{q}(x)\tau_{-y}\chi_{B_\eta}(x)d\mu(x)\\
&\leq& \frac{1}{\mu(B_\eta)}\int_{\mathbb R}\left|f\right|^{q}(x)\tau_{-y}\chi_{B_1}(x)d\mu(x)=0.
\end{eqnarray*}
It comes from (\ref{translate2}) that
\begin{equation*}
0=\lim_{\eta\rightarrow 0}\frac{1}{\mu(B_\eta)}\int_{B_\eta}(\tau_{y}\left|f\right|^{q})(x)d\mu(x)=\left|f\right|^{q}(y).
\end{equation*} 
Thus $f=0$ $\mu$-a.e.

For the triangular inequality, we have 
\begin{equation*}
\left\|f+g\right\|_{q,p}
=\left\|\left[\int_{\mathbb R}\left(\left|f(x)+g(x)\right|(\tau_{(\cdot)}\chi_{B_1}(-x))^{\frac{1}{q}}\right)^qd\mu(x)\right]^{\frac{1}{q}}\right\|_{p}
\end{equation*}
for $f,g\in(L^q,L^p)(\mu)$, and the result follows by applying respectively the triangular inequality in $L^q(\mu)$ and $L^p(\mu)$.
\item If $q=p=\infty$ then there is nothing to prove since $(L^\infty,L^\infty)(\mu)=L^\infty(\mu)$.
\end{enumerate}
Positive homogeneity is immediate.
\end{proof}

The next result is a generalization of \cite[Proposition 1.2.8]{Fe}. The proof is only an adaptation of that given there. We give it for the sake of completeness.

\begin{prop} Let $1\leq  p,q\leq \infty$. The space $ ((L^{q}, L^{p}) (\mu),\left\| \cdot \right\|_{q,p} ) $ is a complex Banach space.
\end{prop}
\begin{proof} Let $(f_{n})_{n \ge 1}$ be a sequence of elements of $(L^{q}, L^{p}) (\mu)$ such that  
	\begin{equation}
	\sum\limits_{n \ge 1} \left\| f_{n} \right\|_{q,p} < +\infty .
	\label{banachs}
	\end{equation}
	
$1^{rst}$ case:  we suppose that $q< \infty$.

We have 
		\begin{eqnarray*}
			\sum\limits_{n \ge 1} \left\| f_{n} \right\|_{q,p} 
			&=& \sum\limits_{n \ge 1} \left\| \left[ \int_{\mathbb R}\left[\left|f_{n}\right|(x)\chi_{B(y,1)}(x)\tau^{\frac{1}{q}}_{-y}\chi_{B_1}(x)\right]^{q}d\mu(x)\right]^{\frac{1}{q}} \right\|_{p} \\
			&=& \sum\limits_{n \ge 1}  \left\|  \left\| f_{n} \chi_{B(y,1)}\tau^{\frac{1}{q}}_{-y}\chi_{B_1} \right\|_{q}
			\right\|_{p}<\infty.	
		\end{eqnarray*}
		Since $L^p(\mu)$ is a Banach space, there exists a measurable subset $N\subset\mathbb R$ with $\mu(N)=0$ such that  
		$$
		\sum\limits_{n \ge 1}  \left\| f_{n} \chi_{B(y,1)}\tau^{\frac{1}{q}}_{-y}\chi_{B_1} \right\|_{q}  < +\infty, \quad  y \in \mathbb{R}\setminus N  .                
		$$ 
		Hence  the series $ \sum\limits_{n \ge 1}  f_{n} \chi_{B(y,1)}\tau^{\frac{1}{q}}_{-y}\chi_{B_1}  $ converges in the Banach space $L^{q} (\mu) $ for all $  y \in \mathbb{R}\setminus N  $.
		
		Put 
		$$E=  \left\{x \in \mathbb{R} \: /  \chi_{B(y,1)} (x) \tau_{-y}\chi_{B_1}(x) \ne 0\text{ for some }y \in \mathbb{R}\setminus N \right\}$$
		
		and  $$f(x) = 
		\left\{
		\begin{array}{lll}
		\sum\limits_{n \ge 1}  f_{n}  (x)& \text{ if }&  x \in E, \\
		0          &\text{ if }& x\notin E.
		\end{array}
		\right.
		$$
		On one hand we have: 
		\begin{eqnarray*}
			\left\| f \right\|_{q,p} 
			&=& \left\| \left[ \int_{\mathbb R}\left|f(x)\chi_{B(y,1)}(x)\tau^{\frac{1}{q}}_{-y}\chi_{B_1}(x)\right|^{q}d\mu(x)\right]^{\frac{1}{q}} \right\Vert_{p} \\
			&=& \left\| \left[ \int_{\left\{
				\begin{array}{ll}
					x \in B(y,1) / \tau_{-y}\chi_{B_1}(x) \ne 0
				\end{array}
				\right\}}\left|\sum\limits_{n \ge 1}  f_{n}  (x) \tau^{\frac{1}{q}}_{-y}\chi_{B_1}(x)\right|^{q}d\mu(x)\right]^{\frac{1}{q}} \right\|_{p} \\
			&\leq& \left\| \left[ \int_{\mathbb R}\left|\sum\limits_{n \ge 1}  f_{n}  (x) \tau^{\frac{1}{q}}_{-y}\chi_{B_1}(x)\right|^{q}d\mu(x)\right]^{\frac{1}{q}} \right\|_{p} 
		\leq \sum\limits_{n \ge 1} \left\Vert   \left\Vert f_{n} \chi_{B(y,1)}\tau^{\frac{1}{q}}_{-y}\chi_{B_1} \right\Vert_{q}
		\right\Vert_{p}<\infty.
		\end{eqnarray*}
		On the other hand, we have 
		\begin{eqnarray*}
			\Vert f_- \sum\limits_{i = 1}^{n} f_{i} \Vert_{q,p} 
			&=& \left\Vert \left[ \int_{\left\{
				\begin{array}{ll}
					x \in B(y,1) / \tau_{-y}\chi_{B_1}(x) \ne 0
				\end{array}
				\right\}}\left|\sum\limits_{i> n}  f_{i}  (x) \tau^{\frac{1}{q}}_{-y}\chi_{B_1}(x)\right|^{q}d\mu(x)\right]^{\frac{1}{q}} \right\Vert_{p} \\
			&\leq& \left\Vert \left[ \int_{\mathbb R}\left|\sum\limits_{i> n}  f_{i}  (x) \tau^{\frac{1}{q}}_{-y}\chi_{B_1}(x)\right|^{q}d\mu(x)\right]^{\frac{1}{q}} \right\Vert_{p} 
		%
		\leq \sum\limits_{i>n} \Vert f_{i} \Vert_{q,p}.
		\end{eqnarray*}
		It comes from Estimate (\ref{banachs}), that  $\lim\limits_{n \to  \infty} \sum\limits_{i>n} \Vert f_{i} \Vert_{q,p} = 0$. 
		
$2^{nd}$ case: we suppose $q =  \infty$.

		Since $\sum\limits_{n \ge 1} \Vert f_{n} \Vert_{ + \infty,p} = \sum\limits_{n \ge 1} \left\|\left\|f\chi_{B(y,1)}\right\|_{+\infty}\right\|_p < + \infty$, there exists a $ \mu$-measurable subset $N$ of $\mathbb{R}$ such that 
		
		$$
		\left\{
		\begin{array}{ll}
		\mu ( N) = 0, \\
		\sum\limits_{n \ge 1}  \Vert f_{n} \chi_{B(y,1)} \Vert_{+\infty } < \infty , \quad y \in \mathbb{R}\setminus N  .                   
		\end{array}
		\right.
		$$
		Hence for all $  y \in \mathbb{R}\setminus N $, the series  $ \sum\limits_{n \ge 1}  f_{n} \chi_{B(y,1)} $ converges a.e on $\mathbb R$.  
		Since for all $y_{1},y_{2}\in \mathbb{R}\setminus N$ such that $B(y_{1},1) \cap B(y_{2},1) \ne \emptyset$, we have  
	$$\sum\limits_{n \ge 1}   f_{n} (x) \chi_{B(y_{1},1)} (x) = \sum\limits_{n \ge 1}   f_{n} (x) \chi_{B(y_{2},1)} (x),$$
	for $ \mu $-almost every $x \in B(y_{1},1) \cap B(y_{2},1)$, we put $f(x) = \sum\limits_{n \ge 1}   f_{n} (x) $. The function $f$ is well defined almost everywhere on $\mathbb R$ and
	
		$$\Vert f \Vert_{ + \infty,p} =  \left\|\left\|f\chi_{B(y,1)}\right\|_{+\infty}\right\|_p =  \left\|\left\|\sum\limits_{n \ge 1}   f_{n} \chi_{B(y,1)}\right\|_{+\infty}\right\|_p
		\sum\limits_{n \ge 1} \Vert f_{n} \Vert_{ + \infty,p} < \infty.$$
		We also have 
		$$\Vert f \chi_{B(y,1)} - \sum\limits_{i = 1}^{n} f_{i} \chi_{B(y,1)}\Vert_{ + \infty} = \Vert \sum\limits_{i >n} f_{i} \chi_{B(y,1)}\Vert_{ + \infty} \leq \sum\limits_{i >n} \Vert  f_{i} \chi_{B(y,1)}\Vert_{ + \infty}$$
		for $\mu$-almost every $y \in \mathbb{R}$ and for all positive integers $n$. Thus  
		$$\left\Vert f_- \sum\limits_{i = 1}^{n} f_{i} \right\Vert_{ + \infty,p} \leq   \sum\limits_{i> n} \left\Vert  f_{i} \right\Vert_{ + \infty,p} ,  $$
	%
	which tends to zero as $n$ goes to infinity. 
\end{proof}

It comes from the definition of the Dunkl convolution that for $1\leq q<\infty$, we have
$$\int_{B_1}(\tau_{y}\left|f\right|^{q})(x)d\mu(x)=\left(\int_{\mathbb R}(\tau_{y}\left|f\right|^{q})(x)\chi_{B_1}(-x)d\mu(x)\right)
=\left|f\right|^{q}\ast_{\kappa}\chi_{1}(y).$$
This observation couple with Young inequality in Dunkl Lebesgue spaces, allow us to prove the following result. 
\begin{prop}\label{inclaqp}Let $1\leq q\leq s\leq p\leq\infty$. We have
\begin{equation}
\left\|f\right\|_{q,p}\leq  4^{\frac{1}{q}}\left\|f\right\|_{s}\mu(B_1)^{\frac{1}{p}-\frac{1}{s}+\frac{1}{q}}.
\end{equation}
\end{prop}
\begin{proof}
Let $1\leq q\leq s\leq p<\infty$. We have 
$$\frac{ps}{qs-pq+ps}\geq 1\text{ and }\frac{1}{\frac{s}{q}}+\frac{1}{\frac{ps}{qs-pq+p\alpha}}=\frac{1}{\frac{p}{q}}+1$$
so that
\begin{equation}
\left\|f\right\|^{q}_{q,p}=\left\|\left|f\right|^{q}\ast_{\kappa}\chi_{B_1}\right\|_{\frac{p}{q}}\leq 4\left\|f\right\|^{q}_{s}\mu(B_1)^{\frac{qs-qp+ps}{ps}}.\label{LsinLqp}
\end{equation}
according to Proposition \ref{young}. 
\end{proof}
The above result allows us to say that $L^{s}(\mu)\subset\left(L^{q},L^{p}\right)(\mu)$ provided  $1\leq q\leq s\leq p<\infty$.

We also have, as in the classical case, that the family of spaces $(L^q, L^p) (\mu)$ is decreasing with respect to the  $ q $ power. More precisely we have the following result.
\begin{prop}\label{include1}
Let $1\leq q_1\leq q_2\leq p\leq\infty$. For any locally integrable function $f$, we have that
\begin{equation}\left\|f\right\|_{q_1,p}\leq\mu(B_1)^{\frac{1}{q_1}-\frac{1}{q_2}}\left\|f\right\|_{q_2,p}\label{decrease}
\end{equation}
\end{prop}
\begin{proof}
Let $f\in L^0(\mu)$. 

We suppose that $q_{1} < q_{2} < + \infty$. H\"older inequality, and \cite[Theorem 6.3.3]{DX} 
allow us to write that 
\begin{eqnarray*}
\left[\int_{\mathbb R}\tau_{y}\left|f\right|^{q_{1}}(x)\chi_{B_1}(x)d\mu(x)\right]^{\frac{1}{q_{1}}}	&\leq&  \left[ \int_{\mathbb R}(\left|f\right|^{q_{2}}(x) \tau_{-y}\chi_{B_1}(x)d\mu(x)\right]^{\frac{1}{q_{2}}} \left[\int_{\mathbb R}\tau_{-y}\chi_{B_1}(x)d\mu \right]^{\frac{1}{q_{1}}-\frac{1}{q_{2}}}\\
	&\leq& \left[ \int_{\mathbb R}\left|f\right|^{q_{2}}(x) \tau_{-y}\chi_{B_1}(x)d\mu(x)\right]^{\frac{1}{q_{2}}}  (\mu (B_1))^{\frac{1}{q_{1}}-\frac{1}{q_{2}} } \\	
	&=& (\mu (B_1))^{\frac{1}{q_{1}}-\frac{1}{q_{2}}} \left[\int_{\mathbb R}\tau_{y}\left|f\right|^{q_{2}}(x)\chi_{B_1}(x)d\mu(x)\right]^{\frac{1}{q_{2}}},	
\end{eqnarray*}	
so that taking the $L^p(\mu)$-norm of both sides leads to the desire result.

Suppose  $q_{1} < q_{2} = + \infty$. We assume that $\left\| f \right\|_{ \infty,p} <\infty$ since otherwise there is nothing to prove. We have
\begin{eqnarray*}
\left[\int_{\mathbb R}\tau_{y}\left|f\right|^{q_{1}}(x)\chi_{B_1}(x)d\mu(x)\right]^{\frac{1}{q_{1}}} &=& \left[\int_{ B(y,1) }\left|f\right|^{q_{1}}(x)\tau_{-y}\chi_{B_1}(x)d\mu_(x)\right]^{\frac{1}{q_{1}}} \\
&\leq& \left\| f \chi_{B(y,1)} \right\|_{ \infty } \left[\int_{ \mathbb R }\tau_{-y}\chi_{B_1}(x)d\mu(x)\right]^{\frac{1}{q_{1}}} \\
&=& (\mu (B_1))^{\frac{1}{q_{1}}} \left\| f \chi_{B(y,1)} \right\|_{ \infty }.
\end{eqnarray*}
We also conclude in this case by taking the $L^p(\mu)$-norm of both sides.
\end{proof}
\section{Fofana spaces associated to Dunkl-amalgam}
Let $1\leq q\leq\alpha\leq p\leq\infty$. The space $\left(L^{q},L^{p}\right)^{\alpha}(\mu)$ is defined by
$$\left(L^{q},L^{p}\right)^{\alpha}(\mu)=\left\{f\in L^{0}(\mu):\left\|f\right\|_{q,p,\alpha}<\infty\right\},$$
 where
\begin{equation*}
\left\|f\right\|_{q,p,\alpha}=\sup_{r>0}(\mu(B_r))^{\frac{1}{\alpha}-\frac{1}{q}-\frac{1}{p}}\ _{r}\left\|f\right\|_{q,p},
\end{equation*}
and for $r>0$
 \begin{equation*}
 \ _{r}\left\|f\right\|_{q,p}=\left\{\begin{array}{lll}\left\|\left[\int_{\mathbb R}(\tau_{y}\left|f\right|^{q})\chi_{B_r}(x)d\mu(x)\right]^{\frac{1}{q}}\right\|_p&\text{ if }&q<\infty\\
 \left\|\left\|f\chi_{B(y,r)}\right\|_{\infty}\right\|_p&\text{ if }&q=\infty.
 \end{array}\right.
 \end{equation*}
 It is easy to see that for 
$q\leq\alpha\leq p$ the space $\left(L^{q},L^{p}\right)^{\alpha}(\mu)$ is a complex vector subspace of $(L^{q},L^{p})(\mu)$.
Notice that $(L^{q},L^{\infty})^{\alpha}(\mu)$ is the Dunkl-type Morrey space as defined in \cite{GM2}.

From the definition of $\left(L^{q},L^{p}\right)^{\alpha}(\mu)$ spaces and the proof of Proposition \ref{norm1}, we deduce that the map $ L^{0}(\mu)\ni f\mapsto\ \left\|f\right\|_{q,p,\alpha}$ define a norm in $\left(L^{q},L^{p}\right)^{\alpha}(\mu)$.

These spaces have been introduced in the classical case; i.e., taking in the above definition the measure of Lebesgue instead of $ \mu $, and the classical translation instead of that of Dunkl,  by I. Fofana in \cite{Fo} where he proved that the space is non trivial if and only if $q\leq \alpha\leq p$ (this is why we are refering to these spaces as Fofana's spaces).  Thus we will always assume that this condition is fulfilled. Some important properties of Fofana's spaces in the classical case  are resumed in the following (see \cite{Fe1}).
\begin{prop}
Let $1\leq q\leq \alpha\leq p\leq \infty$.
\begin{enumerate}
    \item $\left((L^q,L^p)^\alpha(dx),\left\|\cdot\right\|_{(L^q,L^p)^\alpha(dx)}\right)$ is a complex Banach space.
    \item There exists $C>0$ such that
    $$\left\|f\right\|_{(L^q,L^p)(dx)}\leq\left\|f\right\|_{(L^q,L^p)^\alpha(dx)}\leq C\left\|f\right\|_{L^\alpha(dx)}.$$
    \item If $q\leq q_1\leq\alpha\leq p_1\leq p$ then
    $$\left\|f\right\|_{(L^q,L^p)^\alpha(dx)}\leq \left\|f\right\|_{(L^{q_1},L^{p_1})^\alpha(dx)}.$$
    \item If $\alpha\in \left\{q,p\right\}$ then $(L^q,L^p)^\alpha(dx)=L^\alpha(dx)$ with equivalent norms.
\end{enumerate}
\end{prop}

\medskip

In the case of harmonic analysis associated to Dunkl operators, it follows from Proposition \ref{inclaqp} that for $1\leq q\leq\alpha\leq p\leq\infty$, 

\begin{equation}
\left\|f\right\|_{q,p,\alpha}\leq 4^{\frac{1}{q}}\left\|f\right\|_{\alpha},\label{lebesguedansamalg}
\end{equation}
since Relation (\ref{LsinLqp}) is valid for all balls $B_r$ and $\mu(B_r)>0$ for all $r>0$. This proves that $L^\alpha(\mu)$ is continuously embedded in $(L^q,L^p)^\alpha(\mu)$. We also deduce from Proposition \ref{include1} that for $1\leq q_1\leq q_2\leq\alpha\leq p\leq\infty$, we have 
$$\left\|f\right\|_{q_1,p,\alpha}\leq\left\|f\right\|_{q_2,p,\alpha}.$$

\section{Norm inequalities for Hardy-Littlewood type operator}
Let $f\in L^0(\mu)$. The Dunkl-type Hardy-Littlewood maximal function $Mf$ is defined  by
$$M f(x)=\sup_{r>0}(\mu(B(0,r))^{-1}\int_{B(0,r)}\tau_x\left|f\right|(y)d\mu(y) \  x\in\mathbb R.$$
 It is proved in \cite{TX} that $M$ is bounded on $L^p(\mu)$ for $1<p\leq \infty$ and it is of weak type $(1,1)$.
	
We have the following norms inequalities for the Dunkl-type Hardy-Littlewood  maximal function, in the setting of Dunkl-Fofana spaces.

\begin{thm}\label{maxi}
Let $1< q\leq \alpha\leq p\leq \infty$. 
There exists a constant $C>0$ such that
$$\left\|Mf\right\|_{q,p,\alpha}\leq C\left\|f\right\|_{q,p,\alpha}, \ \ f\in L_{loc}^{q}(\mu).$$
\end{thm}

\medskip

Let $1\leq q\leq\alpha\leq \infty$. For a locally $\mu$-measurable function  $f$, we put
$$ \Vert f \Vert_{(L^{1,+\infty}, L^{p})^{\alpha} (\mu)} = \sup\limits_{r > 0} (\mu (B(0,r)))^{\frac{1}{\alpha} - 1 - \frac{1}{p}} \Vert \Vert f \tau_{-y}\chi_{B_r} \Vert_{L^{1,+\infty}(\mu)   }   \Vert_{p},$$
where   $L^{1,+\infty}(\mu)$ denotes the weighted weak-Lebesgue space,  consists of $f\in L^0(\mu)$ 
satisfying the condition 	
	$$ \Vert f \Vert_{L^{1,+\infty} (\mu)} = \sup\limits_{\lambda >0}	\lambda \mu (  \left\{
	\begin{array}{ll}
	x \in \mathbb{R} : \vert f(x) \vert > \lambda 
	\end{array}
	\right\}  )<\infty.$$

Whe have the following weak-type inequality, when  $q=1$. 
\begin{thm}\label{weakmaxi}
Let $1\leq \alpha\leq p\leq \infty$. There exists  $C>0$ such that 
$$ \Vert M f \Vert_{(L^{1,+\infty}, L^{p})^{\alpha} (\mu)} \leq C \Vert f \Vert_{1,p,\alpha}, \ \ f\in L_{loc}^{1}(\mu).$$
\end{thm}

\medskip

Notice that $(\mathbb R,\left|\cdot\right|,\mu)$ is a space of homogeneous type in the sense of Coifman and Stein.

Let $X$ be a set. A map  $d : X \times X\rightarrow\left[0,\infty\right)$ is called a quasi-distance on
$X$ if the following conditions are satisfied:
\begin{enumerate}
\item for every $x$ and $y$ in $X$, $d(x, y) = 0$ if and only if $x = y$,
\item for every $x$ and $y$ in $X$, $d(x, y) = d(y, x)$,
\item there exists a constant $\mathfrak c$ such that $d(x, y) \leq \mathfrak {c}(d(x, z)+ d(z, y))$ for every
$x,\  y$ and $z$ in $X$.\label{constdoub}
\end{enumerate}

Let $x\in X$ and $r>0$. The ball of center $x$ and radius $r$ is the subset $I(x, r) = \left\{y : d(x, y) < r\right\}$ of $X$. Fix $x\in X$. The family $\left\lbrace I(x,r)\right\rbrace _{r>0}$  form a basis of neighborhoods of $x\in X$ for the topology induced by $d$. 

Let $X$ be a set endowed with a quasi-distance $d$ and  a non-negative measure  $\mu$ defined on a $\sigma$-algebra of subsets of $X$ which contains the $d$-open subsets and the balls $I(x,r)$. Assume that there exists a finite constant $C$, such that 
\begin{equation}
 0 < \mu(I(x,2r))\leq C\mu(I(x, r)) <\infty\ x\in X\text{ and } r>0, \label{0.2}
\end{equation}
holds for every $x\in X$ and $r>0$. A set $X$ with a quasi-distance $d$ and a measure $\mu$ satisfying the above conditions, will be called a space of homogeneous type and denoted by $(X,d,\mu)$.  
If $C'_{\mu}$ is the smallest constant $C$ for which (\ref{0.2}) 
holds, then $D_{\mu}=\log _{2}C'_{\mu}$ is called the
doubling order of $\mu$. It is known (see \cite{sw}) that  for all balls $I_{2}\subset I_{1}$ 
\begin{equation}
\frac{\mu \left( I_{1}\right) }{\mu \left( I_{2}\right) }\leq C_{\mu}\left( \frac{r\left(
I_{1}\right) }{r\left( I_{2}\right) }\right) ^{D_{\mu}},
\label{0.05}
\end{equation}
where $r(I)$ denotes the radius of the ball $I$, $C_{\mu}=C'_{\mu}(2\mathfrak {c})^{D_{\mu}}$ and $\mathfrak c \geq 1$ the constant in (\ref{constdoub}).

\medskip

In the sequel we assume that $X=\left( X,d,\mu \right) $ is a fixed space of homogeneous type and:
\begin{itemize}
\item $(X,d)$ is separable, 
\item $\mu(X)=\infty,$
\item $I(x,R)\setminus I(x,r)\neq \emptyset,\ 0<r<R<\infty$ and $x\in X$.
\end{itemize}
As proved in \cite{W}, the last assumption implies that there exist two constants $\tilde{C}_{\mu}>0$ and $\delta_{\mu}>0$ such that for all balls $I_{2}\subset I_{1}$ of $X$
 \begin{equation}
 \frac{\mu(I_{1})}{\mu(I_{2})}\geq \tilde{C}_{\mu}\left(\frac{r(I_{1})}{r(I_{2})}\right)^{\delta_{\mu}}.\label{revd}
 \end{equation}
 For $ y \in \mathbb R $ and $ r> 0 $, we have the following relation between $ \mu(B (y, r)) $ and $ \mu(I (y, r)) $, when $ \mu $ is the weighted Lebesgue measure defined in Relation (\ref{*}).

\begin{lem}\label{lem 8}
	Let  $x \in \mathbb{R}$. For all $r>0$ we have  $$  \mu (B(x,r)) \leq 2 \mu (I(x,r)).$$
\end{lem}

\begin{proof}
	Let $x \in \mathbb{R}$ and  $r>0$. We have 
	\begin{equation*}
	B(x,r) =\left\{\begin{array}{lll} ] - \vert x \vert -r , 0 [ \cup ] 0,  \vert x \vert +r [ &\text{ if }&  \vert x \vert \leq r \\
	] - \vert x \vert -r , - \vert x \vert + r  [ \cup ] \vert x \vert - r ,  \vert x \vert +r [ &\text{ if }& \vert x \vert > r
	\end{array}\right.
	\end{equation*}
	
	\begin{enumerate}
		\item If $x \ge r$ then
		\begin{eqnarray*}
			\mu(B(x,r)) 
			&=& c_{\k}(\k+1)^{-1}[(x+r)^{2 \k+2} - ( x  - r )^{2\k+2}  ]
		\end{eqnarray*}
		so that 
		 $  \mu(B(x,r)) = 2 \mu (I(x,r)).$
		\item We suppose now that $x \leq r$. 
		\begin{enumerate}
			\item For $0 <x < r$,  we have  $$\mu (B(x,r)) 
			=c_{\k} ( \k+1)^{-1} (  x  + r )^{2\k+2} .$$
			and
			\begin{equation*}
				\mu (I(x,r)) 
				= \frac{c_{\k}}{2\k+2} [ ( r-x )^{2\k+2} + (x + r )^{2\k+2}  ],
			\end{equation*}
			which leads to $  \mu (B(x,r)) < 2 \mu (I(x,r)).$
			\item Suppose $-r < x <0$. We have  $\mu (B(x,r)) = \mu (B(-x,r))$ and $$\mu (I(x,r)) =  \frac{c_{\k}}{2\k+2} [ ( r-x )^{2\k+2} + (x + r )^{2\k+2}  ]= \mu (I(-x,r)),$$
			so that  $  \mu (B(-x,r)) < 2 \mu (I(-x,r)).$ 
			\item Suppose $x<-r$. We have
			$$\mu(B(x,r)) =  c_{\k} ( \k+1)^{-1} [ ( \vert x \vert + r )^{2\k+2} - (\vert x \vert - r )^{2\k+2}  ] = \mu (B(-x,r)) $$ 
%
			and $$\mu (I(x,r)) =  \frac{c_{\k}}{2\k+2} [ ( r-x )^{2\k+2} + (x + r )^{2\k+2}  ].$$
			Thus $2\mu (I(x,r)) > 2 \mu(I(-x,r)) = \mu (B(-x,r)) = \mu (B(x,r))$.
		\end{enumerate}
		\end{enumerate}
	This completes the proof.	
\end{proof}

It comes from the above lemma that there exists $C>0$ such that for all $x \in \mathbb{R}$ and $r>0$, we have 
	\begin{equation}
	    \mu(I(x,2r))\leq C\mu(I(x,r)),\label{rel12}
	\end{equation}
thanks to \cite{De}.

In \cite{FFK2}, the space $(L^q,L^p)^\alpha(X,d,\mu)$ is defined for $1\leq q\leq\alpha\leq p\leq\infty$ as the set of $\mu$-measurable functions $f$ satisfying $\left\|f\right\|_{(L^q,L^p)^\alpha(X,d,\mu)}<\infty$, where 
$\left\|f\right\|_{(L^q,L^p)^\alpha(X,d,\mu)}=\sup_{r>0}\ _{r}\left\|f\right\|_{(L^q,L^p)^\alpha(X,d,\mu)}$ with 
$$\ _{r}\left\|f\right\|_{(L^q,L^p)^\alpha(X)}=\left\{\begin{array}{lll}\left[\int_{X}\left(\mu(I(y,r))^{\frac{1}{\alpha}-\frac{1}{q}-\frac{1}{p}}\left\|f\chi_{I(y,r)}\right\|_{L^q(X,d,\mu)}\right)^{p}d\mu(y)\right]^{\frac{1}{p}}&\text{ if }&p<\infty\\
\sup_{y\in X}\mu(I(y,r))^{\frac{1}{\alpha}-\frac{1}{q}}\left\|f\chi_{I(y,r)}\right\|_{L^q(X,d,\mu)}&\text{ if }&p=\infty
\end{array}\right.$$

It is proved in \cite[Theorem 2.9]{FFK2} that if $\alpha\in\left\lbrace q,p\right\rbrace $ then 
\begin{equation}
\Vert\cdot\Vert_{\alpha}\lsim \Vert\cdot\Vert_{(L^q,L^p)^\alpha(\mathbb R,\vert\cdot \vert,\mu)}.\label{inclusionL}
\end{equation}
where $\mu$ is de measure defined by  (\ref{*}). For the same $\mu$, we have the following result:
 \begin{lem} \label{lem 5}
Let $x \in \mathbb{R}$. For all $r>0$ $$ \mu(I(0,r))  \leq 2 \mu(I(x,r)).$$
\end{lem}

\begin{proof}
Let's fix  $x \in \mathbb{R}$ and $r>0$.
\begin{enumerate}
\item We suppose that $x < r$. We have 
$$2 \mu(I(x,r)) = c_{\k} ( \k+1)^{-1} [ ( x  + r )^{2\k+2} + (x - r )^{2\k+2}  ]$$
as we can see in the proof of Lemma \ref{lem 8}. Thus, $ \mu(I(0,r))  \leq 2 \mu(I(x,r)).$

\item We suppose now $x \ge r$. 
According to the proof of Lemma \ref{lem 8} once more, we have 
$$2 \mu(I(x,r)) = c_{\k} ( \k+1)^{-1} [ ( x  + r )^{2\k+2} - (x - r )^{2\k+2}  ]$$
Put  $A=  ( x  + r )^{2\k+2} - (x - r )^{2\k+2} - r^{2\k+2} = r^{2\k+2} [ ( \frac{x}{r}  + 1 )^{2\k+2} - (1 - \frac{x}{r})^{2\k+2} - 1]=\phi(\frac{x}{r})$ with 
$\phi (t) =r^{2\k+2}[ ( t  + 1 )^{2\k+2} - (1 - t)^{2\k+2} - 1]$. The function $\phi$ is increasing for  $t\geq 1$ and $\frac{x}{r}\geq 1$. Thus $A\geq \phi(1)= 2^{2\k+2} -1 >0$. 

Hence 
$2 \mu(I(x,r)) > \mu(I(0,r)) $. 
\end{enumerate}
\end{proof}

It comes from the above lemma and a result in \cite{MH} that for $1\leq q\leq p\leq\infty$ and $\mu$ as in (\ref{*}), the spaces $(L^q,L^p)^\alpha(\mu)\subset (L^q,L^p)^\alpha(\mathbb R,\vert\cdot\vert,\mu)$. More precisely, we have 
$\Vert\cdot\Vert_{(L^q,L^p)^\alpha(\mathbb R,\vert\cdot\vert,\mu)}\leq \Vert\cdot\Vert_{q,p,\alpha}$. 
However, we still have the following analogue of the Inequality (\ref{inclusionL}) in $(L^q,L^p)^\alpha(\mu)$ spaces.
\begin{prop}
Let $1\leq q\leq\alpha\leq p\leq\infty$.
 If $\alpha\in\left\{q,p\right\}$ then $\left(L^{q},L^{p}\right)^{\alpha}(\mu)=L^{\alpha}(\mu)$.
\end{prop}

\begin{proof}
Let $1\leq q\leq\alpha\leq p\leq\infty$.  It comes from Inequality (\ref{lebesguedansamalg})  that 
$L^{\alpha}(\mu)\subset\left(L^{q},L^{p}\right)^{\alpha}(\mu)$. So, all we have to prove is the reverse inclusion.

Let $f\in \left(L^{q},L^{p}\right)^{\alpha}(\mu)$. We have 
$$\Vert f\Vert_{(L^q,L^p)^\alpha(\mathbb R,\vert\cdot \vert,\mu)}\leq \Vert f\Vert_{q,p,\alpha}<\infty.$$
Therefore $f\in  (L^q,L^p)^\alpha(\mathbb R,\vert\cdot\vert,\mu)$ and 
 the result follows from Inequality (\ref{inclusionL}).
\end{proof}

 For a $\mu$-measurable function $f$ in the space of homogeneous type $(X,d,\mu)$, we define the Hardy-Littlewood maximal function $M_\mu f$ by
$$ M_\mu f(x)=\sup_{r>0}\mu(I(x,r))^{-1}\int_{I(x,r)}\left|f(y)\right|d\mu(y).$$
The following result is an extension of the analogue on Euclidean space given in \cite{Fe1}. Its proof is just an adaptation of the one given there.
\begin{prop}\label{contmax}
Let $1<q\leq\alpha\leq p\leq\infty$. There exists $C>0$ such that 
$$\left\| M_\mu f\right\|_{(L^q,L^p)^\alpha(X,d,\mu)}\leq C \left\|f\right\|_{(L^q,L^p)^\alpha(X,d,\mu)}$$
for all $f\in (L^q,L^p)^\alpha(X,d,\mu)$. 
\end{prop}
For the proof, we will use the following well known lemma. The proof is given just for completeness.
\begin{lem}\label{maxcha}There exists a constant $C>0$ such that for all $r>0$,
$$ M_\mu \chi_{I(y,r)}(x)\leq C\frac{\mu(I(y,r)}{\mu(I(y,d(x,y))},\ x\neq y\in X.$$
\end{lem}
\begin{proof}
Let $r>0$, $x,y\in X$. We have 
$$M_{\mu}(\chi_{I(y,r)})(x)= \sup_{\rho>0} \mu(I(x,\rho))^{-1} \int_{I(x,\rho)} \chi_{I(y,r)} (z) d\mu (z) = \sup_{\rho>0} \frac{\mu(I(y,r) \cap I(x,\rho) )}{\mu(I(x,\rho))} \leq 1.$$
Let's fix $\rho>0$. If $d(x,y) \leq \rho$ then the result is immediate. Assume $\rho>d(x,y)$. We also assume that $I(x,r)\cap I(y,\rho)\neq\emptyset$, since otherwise there is nothing to prove.
We have $d(x,y) \leq \mathfrak c (d(x,u) + d(u,y)) < \mathfrak c ( \rho + r )\leq 2\mathfrak c \max(\rho,r)$, where $u\in I(x,r)\cap I(y,\rho)$. 
 It comes from the doubling condition that 
$$\mu((I(x,d(x,y))\approx\mu(I(y,d(x,y)) \leq \left\{\begin{array}{lll}\mu(I(y,2\mathfrak c  r)) \leq C \mu (I(y, r))&\text{ if }&\rho\leq r\\
\mu(I(x,2\mathfrak c \rho)) \leq C \mu (I(x, \rho))&\text{ if }&\rho> r\end{array}\right.
.$$
Which ends the proof.
\end{proof}

\begin{proof}[Proof of Proposition \ref{contmax}]
For all $y\in X$ and $r>0$, we have
\begin{equation}
\int_{I(y,r)}M_\mu(f)^q(x)d\mu(x)\lsim\int_{X}\left|f(x)\right|^{q}M_\mu(\chi_{I(y,r)})(x)d\mu(x),\label{ffs}
\end{equation}
thanks to  \cite[Theorem 5.1]{PW}. 

Let's  fix $y\in X$ and $r>0$. It comes from (\ref{ffs}) that
\begin{eqnarray*}
\left\| M_\mu(f)\chi_{I(y,r)}\right\|^{q}_{L^q(X,d,\mu)}&\lsim &\int_{I(y,2\mathfrak c r)}\left|f(x)\right|^{q} M_\mu(\chi_{I(y,r)})(x)d\mu(x)\\
&+&\int_{I(y,2\mathfrak c r)^{c}}\left|f(x)\right|^{q} M_\mu(\chi_{I(y,r)})(x)d\mu(x)=I+II.
\end{eqnarray*}
But, $I\lsim\left\|f\chi_{I(y,2\mathfrak c r)}\right\|^{q}_{L^q(X,d,\mu)}$. It remains to estimate II. Applying Lemma \ref{maxcha}, we have 
\begin{eqnarray*}II&\lsim&\sum^{\infty}_{k=1}\int_{2^k\mathfrak c r\leq d(x,y)<2^{k+1}\mathfrak c r}\left|f(x)\right|^{q}\frac{\mu(I(y,r))}{\mu(I(y,d(x,y)))}d\mu(x)\\
&\lsim&\sum^{\infty}_{k=1}\int_{2^k\mathfrak c r\leq d(x,y)<2^{k+1}\mathfrak c r}\left|f(x)\right|^{q}\frac{\mu(I(y,r))}{\mu(I(y,2^k\mathfrak c r))}d\mu(x)\\
&\lsim&\sum^{\infty}_{k=1}\frac{\mu(I(y,r))}{\mu(I(y,2^k\mathfrak c r))}\left\|f\chi_{I(y,2^{k+1}\mathfrak c r)}\right\|^{q}_{L^q(X,d,\mu)}.
\end{eqnarray*}
Therefore 
\begin{equation}
\begin{array}{lll}\left\|M_\mu(f)\chi_{I(y,r)}\right\|^{q}_{L^q(X,d,\mu)}&\lsim& \left\|f\chi_{I(y,2\mathfrak c r)}\right\|^{q}_{L^q(X,d,\mu)}\\
&+&\sum^{\infty}_{k=1}\frac{\mu(I(y,r))}{\mu(I(y,2^k\mathfrak c r))}\left\|f\chi_{I(y,2^{k+1}\mathfrak c r)}\right\|^{q}_{L^q(X,d,\mu)}.\end{array}\label{equa1}
\end{equation}
Multiplying both sides of (\ref{equa1}) by $\mu(I(y,r))^{q(\frac{1}{\alpha}-\frac{1}{p}-\frac{1}{q})}$ and using the reverse doubling condition (\ref{revd}), we obtain 
$$\begin{aligned}&\mu(I(y,r))^{q(\frac{1}{\alpha}-\frac{1}{p}-\frac{1}{q})}\left\|M_\mu(f)\chi_{I(y,r)}\right\|^{q}_{L^q(X,d,\mu)}\\
&\ \ \ \ \ \ \ \lsim\mu(I(y,2\mathfrak c r))^{q(\frac{1}{\alpha}-\frac{1}{p}-\frac{1}{q})}\left\|f\chi_{I(y,2\mathfrak c r)}\right\|^{q}_{L^q(X,d,\mu)}\\
&\ \ \ \ \ \ \ +\sum^{\infty}_{k=1}\left(\frac{1}{2^{\delta_\mu q(\frac{1}{\alpha}-\frac{1}{p})}}\right)^{k}\mu(I(y,2^{k+1}\mathfrak c r))^{q(\frac{1}{\alpha}-\frac{1}{p}-\frac{1}{q})}\left\|f\chi_{I(y,2^{k+1}\mathfrak c r)}\right\|^{q}_{L^q(X,d,\mu)},\end{aligned} $$
for all $y\in X$ and $r>0$. Thus, taking the $L^{\frac{p}{q}}$-norms of both sides of the above estimation leads to
$$ _{r}\left\| M_\mu f\right\|^{q}_{(L^q,L^p)^\alpha(X,d,\mu)} \lsim\left\|f\right\|^{q}_{(L^q,L^p)^\alpha(X,d,\mu)}+\sum^{\infty}_{k=1}\left(\frac{1}{2^{\delta_\mu q(\frac{1}{\alpha}-\frac{1}{p})}}\right)^{k}\ _{2^{k+1}Kr}\left\|f\right\|^{q}_{(L^q,L^p)^\alpha(X,d,\mu)}.$$
The result follows from the definition of the $(L^q,L^p)^\alpha(X)$ spaces and the convergence of the series $\sum^{\infty}_{k=1}\left(\frac{1}{2^{\delta_\mu q(\frac{1}{\alpha}-\frac{1}{p})}}\right)^{k}$.
\end{proof}

\section{Proof of Theorems \ref{maxi} and \ref{weakmaxi}}
For the proof of these theorems, we need some lemmas.
\begin{lem} \label{lem 6}
Let $x \in \mathbb{R}$, $r>0$. For all  $y \in \mathbb{R}$
$$M_{\mu} (\tau_{x} \chi_{I(0,r)}  ) (y) = M_{\mu} ( \chi_{I(x,r)} ) (y). $$
\end{lem}

\begin{proof}
Let $x \in \mathbb{R}$, $r>0$. For $y \in \mathbb{R}$,
\begin{eqnarray*}
 \int_{I(y,\rho)} \tau_{x} \chi_{I(0,r)} (z) d\mu (z) 
 &=& \int_{\mathbb{R}}  \chi_{I(0,r)} (z) \tau_{-x}\chi_{I(y,\rho)}(z) d\mu (z) = \int_{\mathbb{R}}  \chi_{I(z,r)} (0) \tau_{-x}\chi_{I(y,\rho)}(z) d\mu (z) \\
 &=&  \int_{\mathbb{R}}  \tau_{x}\chi_{I(z,r)} (0) \chi_{I(y,\rho)}(z) d\mu (z)  
= 
\int_{I(y,\rho)}  \chi_{I(x,r)} (z)  d\mu (z). 
\end{eqnarray*}
The result follows.
\end{proof}

\begin{lem} \label{lem 1}
Let $r>0$ and $y \in \mathbb{R}$. Then $$ \chi_{B(y,r)} \leq \chi_{I(-y,3r)} + \chi_{I(y,3r)}.$$
\end{lem}   
 
\begin{proof}
Let $r>0$ and $y \in \mathbb{R}$.
We have $B(y,r) \subset I(y,3r) \cup I(-y,3r)$. 
\end{proof}  
For all locally $\mu$-integrable functions $f$, we  put
 $$\widetilde{M } f(x) = \sup\limits_{\rho > 0} \frac{1}{\mu (B(x,\rho))} \int_{B(x,\rho)} \vert f \vert (z) d\mu (z) \quad  x \in \mathbb{R}.$$
 We have the following result :
 
\begin{lem} \label{lem3}
Let $f \in L^{1} _{loc}(\mu)$. For all $y \in \mathbb{R}$, $$  M f(y) \approx \widetilde{M } f(y) \approx M_{\mu} f(y).$$
\end{lem}

\begin{proof} Let $f \in L^{1} _{loc}(\mu)$.
 There exist two constants $C_{1}>0$ and $C_{2}>0$ not depending on $f$, such that 
\begin{equation}\label{eq 1}
M f(y) \leq C_{1} \widetilde{M } f(y) \leq C_{2} M_{\mu} f(y),
\end{equation}
for all $ y \in \mathbb{R}$ as we can see in \cite{M}. We also have (see  \cite{MH}) that there exists a constant $C>0$ not depending on $f$ such that 
$$\int_{I(y,\rho)} \vert f(x) \vert d\mu (x) \leq C \int_{B(0,\rho)} \tau_{y} \vert f(x) \vert d\mu (x)  $$
for all  $ \rho > 0$ and $ y \in \mathbb{R}.$ Since we have
$$ 
  \frac{1}{\mu (I(y,\rho))} \int_{I(y,\rho)} \vert f \vert (z) d\mu (z) \leq \frac{2C}{\mu (B(0,\rho))} \int_{B(0,\rho)} \tau_{y} \vert f(x) \vert d\mu (x)  $$
  according to Lemma \ref{lem 5}, the result follows from Relation (\ref{eq 1}).
\end{proof}

\begin{proof}[Proof of Theorem \ref{maxi}] 
We can suppose that  $1<q<\alpha< p\leq\infty$, since otherwise the space is trivial or equal to Lebesgue space. Let  $f\in (L^{q}, L^{p})^{\alpha} (\mu)$.
Fix $y \in \mathbb{R}$ and $r>0$.  It comes from Lemma \ref{lem3}, Lemma \ref{lem 6} and Lemma \ref{lem 1} and the triangular inequality for $L^{q}(\mu)$-norm that 
$$\begin{aligned}
&\left[\int_{\mathbb R}(\tau_{y}\left(Mf\right)^{q})\chi_{B_r}(x)d\mu(x)\right]^{\frac{1}{q}}\\
&\qquad\qquad\qquad\lsim  \left[\int_{\mathbb R} (\widetilde{M } f)^{q}(x)\tau_{-y}\chi_{B_r}(x)d\mu(x)\right]^{\frac{1}{q}} \lsim  \left[\int_{\mathbb R} [\widetilde{M } f(x)\chi_{B(y,r)}(x)]^{q}d\mu(x)\right]^{\frac{1}{q}} \\
&\qquad\qquad\qquad\lsim  \left[\int_{\mathbb R} [\widetilde{M } f(x)\chi_{I(-y,3r)}(x) +\widetilde{M } f(x)\chi_{I(y,3r) }(x))]^{q}d\mu(x)\right]^{\frac{1}{q}} \\
&\qquad\qquad\qquad\lsim \left( \int_{\mathbb R} [\widetilde{M } f(x)\chi_{I(-y,3r)}(x) ]^{q}d\mu(x) \right)^{\frac{1}{q}}   +\left( \int_{\mathbb R} [\widetilde{M } f(x)\chi_{I(y,3r)}(x) ]^{q}d\mu(x) \right)^{\frac{1}{q}}  
\end{aligned}.$$
Making appeal to  \cite[Theorem 2.35]{De}, we have 
$$\left( \int_{\mathbb R} [\widetilde{M } f(x)\chi_{I(-y,3r)}(x) ]^{q}d\mu(x) \right)^{\frac{1}{q}} \lsim\left( \int_{\mathbb R} \vert f(x)\vert^{q} \widetilde{M }\chi_{I(-y,3r)}(x) d\mu(x) \right)^{\frac{1}{q}} $$
and 
$$\left( \int_{\mathbb R} [\widetilde{M } f(x)\chi_{I(y,3r)}(x) ]^{q}d\mu(x) \right)^{\frac{1}{q}} \lsim \left( \int_{\mathbb R} \vert f(x)\vert^{q} \widetilde{M }\chi_{I(y,3r)}(x) d\mu(x) \right)^{\frac{1}{q}} .$$ 
It follows that 
\begin{equation} \label{rel 9}\begin{array}{lll}
\left[\int_{\mathbb R}(\tau_{y}\left(Mf\right)^{q})\chi_{B_r}(x)d\mu(x)\right]^{\frac{1}{q}} &\lsim&  \left( \int_{\mathbb R} \vert f(x)\vert^{q} \widetilde{M }\chi_{I(-y,3r)}(x) d\mu(x) \right)^{\frac{1}{q}} \\
 &+&  \left( \int_{\mathbb R} \vert f(x)\vert^{q} \widetilde{M }\chi_{I(y,3r)}(x) d\mu(x) \right)^{\frac{1}{q}}. \end{array}
\end{equation}
Since Dunkl translation commute with Dunkl Maximal function (see \cite{GEM}), we have  
\begin{eqnarray*}
	 \int_\mathbb{R} \vert f \vert^q (x) \widetilde{M } \chi_{I(-y,3r)} (x) d\mu (x) &\lsim  & \int_\mathbb{R} \vert f \vert^q (x) M_{\mu}  \chi_{I(-y,3r)} (x) d\mu (x) \\
	& \lsim &  \int_\mathbb{R} \vert f \vert^q (x) M_{\mu} ( \tau_{-y} \chi_{I(0,3r)} )(x) d\mu (x) \\
	& \lsim&   \int_\mathbb{R} \vert f \vert^q (x) M ( \tau_{-y} \chi_{I(0,3r)}) (x) d\mu (x) \\
	& \lsim &   \int_\mathbb{R} \vert f \vert^q (x)  \tau_{-y}( M \chi_{I(0,3r)}) (x) d\mu (x) \\
	& \lsim &  \int_\mathbb{R} \tau_{y} \vert f \vert^q (x)  ( M \chi_{I(0,3r)}) (x) d\mu \\
	& \lsim &  \int_\mathbb{R} \tau_{y} \vert f \vert^q (x)  ( M_{\mu} \chi_{I(0,3r)}) (x) d\mu (x) ,
\end{eqnarray*}
according to Lemma \ref{lem3} and Lemma \ref{lem 6}. That is 
\begin{equation}
    \left(\int_\mathbb{R} \vert f \vert^q (x) \widetilde{M } \chi_{I(-y,3r)} (x) d\mu (x) \right)^{\frac{1}{q}}\lsim  \left(\int_\mathbb{R} \tau_{y} \vert f \vert^q (x)  ( M_{\mu} \chi_{I(0,3r)}) (x) d\mu (x)\right)^{\frac{1}{q}}.\label{estimate1}
\end{equation}
The term on the right hand side of (\ref{estimate1}) is less or equal to 
$$C	\left[  \left(\int_{I(0,2r)} \tau_{y}\left|f(x)\right|^q M_{\mu} \chi_{I(0,3r)} (x) d\mu(x)\right)^{\frac{1}{q}} 
	+\left(\int_{I(0,2 r)^{c}} \tau_{y} \left|f(x)\right|^q M_{\mu} \chi_{I(0,3r)} (x) d\mu (x)\right)^{\frac{1}{q}} \right] .$$
Since  $M_{\mu} \chi_{I(0,3r)} (x) = \sup\limits_{\rho > 0} \frac{\mu(I(0,3r) \cap I(x,\rho))}{\mu (I(x,\rho))} \leq 1$ for all  $x \in \mathbb{R}$, we have 
$$  \left(\int_{I(0,2r)} \tau_{y}\left|f(x)\right|^q M_{\mu} \chi_{I(0,3r)} (x) d\mu_{k} (x)\right)^{\frac{1}{q}} \leq \left(\int_{I(0,2r)} \tau_{y} \left|f(x)\right|^q  d\mu (x)\right)^{\frac{1}{q}} $$ 
and 
\begin{eqnarray*}
	\left(\int_{I(0,2 r)^{c}} \tau_{y} \left|f(x)\right|^q M_{\mu} \chi_{I(0,3r)} (x) d\mu (x)\right)^{\frac{1}{q}}  &=& \left(\sum^{\infty}_{i=1}\int_{2^i  r\leq \vert x \vert <2^{i+1} r}\tau_{y} \left|f(x)\right|^q M_{\mu} \chi_{I(0,3r)} (x) d\mu(x)\right)^{\frac{1}{q}} \\
	&\lsim&  \sum^{\infty}_{i=1}\left(\int_{2^i  r\leq \vert x \vert <2^{i+1} r}\tau_{y} \left|f(x)\right|^q \frac{\mu(I(0,r))}{\mu(I(0,\vert x \vert))}d\mu(x)\right)^{\frac{1}{q}}\\
	&\lsim& \sum^{\infty}_{i=1} \frac{\mu(I(0,r))^{\frac{1}{q}}}{\mu(I(0,2^i  r))^{\frac{1}{q}}} \left(\int_{I(0,2^{i+1} r) }\tau_{y} \left|f(x)\right|^q  d\mu(x)\right)^{\frac{1}{q}},
\end{eqnarray*}
thanks to Lemma \ref{maxcha}.
Therefore 
$$\begin{aligned}&\left(\int_\mathbb{R} \vert f \vert^q (x) \widetilde{M } \chi_{I(-y,3r)} (x) d\mu (x) \right)^{\frac{1}{q}}\\
&\ \ \ \ \ \ \ \ \ \lsim\left( \int_{I(0,2r)} \tau_{y} \left|f(x)\right|^q  d\mu (x)\right)^{\frac{1}{q}}  + \sum^{\infty}_{i=1} \frac{\mu(I(0,r))^{\frac{1}{q}}}{\mu(I(0,2^i  r))^{\frac{1}{q}}} \left(\int_{I(0,2^{i+1} r) } \tau_{y} \left|f(x)\right|^q  d\mu(x)\right)^{\frac{1}{q}} .
\end{aligned}$$
Using the same arguments as above, we obtain
$$\begin{aligned}&\left( \int_\mathbb{R} \vert f \vert^q (x) \widetilde{M } \chi_{I(y,3r)} (x) d\mu (x)\right)^{\frac{1}{q}}\\
&\ \ \ \ \ \ \ \ \ \lsim \left(\int_{I(0,2r)} \tau_{-y} \left|f(x)\right|^q  d\mu (x)\right)^{\frac{1}{q}}  + \sum^{\infty}_{i=1} \frac{\mu(I(0,r))^{\frac{1}{q}}}{\mu(I(0,2^i  r))^{\frac{1}{q}}}\left( \int_{I(0,2^{i+1} r) } \tau_{-y} 
\left|f(x)\right|^q  d\mu(x)\right)^{\frac{1}{q}} .
\end{aligned}$$
Taking these two estimates into account, Relation (\ref{rel 9}) becomes 
$$\begin{aligned}&	\left[ \int_{\mathbb R}(\tau_{y}\left(Mf\right)^{q})\chi_{B_r}(x)d\mu(x)\right]^{\frac{1}{q}}\\
&\ \ \ \ \ \ \ \ \ \lsim\left( \int_{I(0,2r)} \tau_{y} \left|f(x)\right|^q  d\mu (x)\right)^{\frac{1}{q}}  + \sum^{\infty}_{i=1} \frac{\mu(I(0,r))^{\frac{1}{q}}}{\mu(I(0,2^i  r))^{\frac{1}{q}}} \left(\int_{I(0,2^{i+1} r) } \tau_{y} \left|f(x)\right|^q  d\mu(x)\right)^{\frac{1}{q}} \\
& \ \ \ \ \ \ \ \ \ \ +  \left(\int_{I(0,2r)} \tau_{-y} \left|f(x)\right|^q  d\mu (x)\right)^{\frac{1}{q}}  + \sum^{\infty}_{i=1} \frac{\mu(I(0,r))^{\frac{1}{q}}}{\mu(I(0,2^i  r))^{\frac{1}{q}}}\left( \int_{I(0,2^{i+1} r) } \tau_{-y} 
\left|f(x)\right|^q  d\mu(x)\right)^{\frac{1}{q}} .
\end{aligned}$$  
Since $\mu$ is symmetric invariant measure, the $L^p(\mu)$-norm of both sides in the above relation yields 
$$\begin{aligned}&\left\Vert \left[\int_{\mathbb R}(\tau_{(\cdot)}\left(Mf\right)^{q})\chi_{B_r}(x)d\mu(x)\right]^{\frac{1}{q}}   \right\Vert_{p} \\
&\qquad\qquad\qquad\qquad \lsim
%
%
%
	\left( \displaystyle \int_\mathbb{R} \left( \int_{I(0,2r)} \tau_{y} \left|f(x)\right|^q  d\mu(x) \right)^{\frac{p}{q}} d\mu (y) \right)^{\frac{1}{p}}\\
&\qquad\qquad\qquad\qquad+\sum^{\infty}_{i=1} \frac{\mu(I(0,r))^{\frac{1}{q}}}{\mu(I(0,2^i  r))^{\frac{1}{q}}} \left( \displaystyle \int_\mathbb{R} \left(\int_{I(0,2^{i+1} r) } \tau_{y} \left|f(x)\right|^q  d\mu(x)  \right)^{\frac{p}{q}} d\mu (y)\right)^{\frac{1}{p}},
	\end{aligned}$$
thanks to Minkowski inequality for integral. 
Multiplying both sides of the inequality by  $(\mu (I(0,r)))^{\frac{1}{\alpha} - \frac{1}{q} - \frac{1}{p}}$, we obtain  

$$\begin{aligned}&(\mu (I(0,r)))^{\frac{1}{\alpha} - \frac{1}{q} - \frac{1}{p}}\left\Vert \left[\int_{\mathbb R}(\tau_{(\cdot)}\left(Mf\right)^{q})\chi_{B_r}(x)d\mu(x)\right]^{\frac{1}{q}}   \right\Vert_{p}\\
&\lsim(\mu (I(0,r)))^{\frac{1}{\alpha} - \frac{1}{q} - \frac{1}{p}}\left( \int_\mathbb{R} \left( \int_{I(0,2r)} \tau_{y} \left|f(x)\right|^q  d\mu (x) \right)^{\frac{p}{q}} d\mu (y) \right)^{\frac{1}{p}}\\
&+\sum^{\infty}_{i=1} \frac{\mu(I(0,r))^{\frac{1}{q}}(\mu (I(0,r)))^{\frac{1}{\alpha} - \frac{1}{q} - \frac{1}{p}}}{\mu(I(0,2^i  r))^{\frac{1}{q}}} \left( \int_\mathbb{R} \left(\int_{I(0,2^{i+1} r) } \tau_{y} \left|f(x)\right|^q  d\mu(x)  \right)^{\frac{p}{q}} d\mu (y)\right)^{\frac{1}{p}} = I +II.
\end{aligned}$$

\begin{equation}I \lsim (\mu (I(0,2r)))^{\frac{1}{\alpha} - \frac{1}{q} - \frac{1}{p}}\left( \int_\mathbb{R} \left( \int_{I(0,2r)} \tau_{y} \left|f(x)\right|^q  d\mu(x) \right)^{\frac{p}{q}} d\mu (y) \right)^{\frac{1}{p}}\lsim \Vert f \Vert_{q,p,\alpha}.\label{relat6}
\end{equation}
as $\frac{1}{\alpha} - \frac{1}{q} - \frac{1}{p} <0$ and $\mu$ is doubling. The second term is less or equal to 
$$\begin{aligned}
& C\sum^{\infty}_{i=1} \frac{(\mu (I(0,r)))^{\frac{1}{\alpha} - \frac{1}{p}}}{(\mu (I(0,2^{i}r)))^{\frac{1}{\alpha} - \frac{1}{p}}} (\mu (I(0,2^{i+1}r)))^{\frac{1}{\alpha} - \frac{1}{q}- \frac{1}{p}}\left( \int_\mathbb{R} \left(\int_{I(0,2^{i+1} r) } \tau_{y} \left|f(x)\right|^q  d\mu(x)  \right)^{\frac{p}{q}} d\mu (y)\right)^{\frac{1}{p}} \\
	&\quad\lsim \sum^{\infty}_{i=1} \left(\frac{1}{2^{\delta_{\mu} (\frac{1}{\alpha} - \frac{1}{p})}} \right)^{i} (\mu (I(0,2^{i+1}r)))^{\frac{1}{\alpha} - \frac{1}{q} - \frac{1}{p}}\left( \int_\mathbb{R} \left(\int_{I(0,2^{i+1} r) } \tau_{y} \left|f(x)\right|^q  d\mu(x)  \right)^{\frac{p}{q}} d\mu (y)\right)^{\frac{1}{p}}
\end{aligned}$$
thanks to Estimate (\ref{revd}).

Hence
\begin{equation} II \lsim\sum^{\infty}_{i=1} \left(\frac{1}{2^{\delta_{\mu} (\frac{1}{\alpha} - \frac{1}{p})}} \right)^{i} \Vert f \Vert_{q,p,\alpha}\lsim \Vert f \Vert_{q,p,\alpha},\label{relat7}
\end{equation}
since the series $\sum^{\infty}_{i=1} \left(\frac{1}{2^{\delta_{\mu} (\frac{1}{\alpha} - \frac{1}{p})}} \right)^{i}$ converges. From Relations (\ref{relat6}) and  (\ref{relat7}), we deduce that 
$$(\mu (I(0,r)))^{\frac{1}{\alpha} - \frac{p}{q} - \frac{1}{p}}	\left\Vert \left[\int_{\mathbb R}(\tau_{(\cdot)}\left(Mf\right)^{q})\chi_{B_r}(x)d\mu(x)\right]^{\frac{1}{q}}   \right\Vert_{p} \lsim \Vert f \Vert_{q,p,\alpha}.$$
We conclude by taking the supremum in the right hand side over all  $r>0$.
\end{proof}

\begin{proof}[Proof of Theorem \ref{weakmaxi}]
 Let  $1 < \alpha < p < +\infty$ and  $f\in (L^{1}, L^{p})^{\alpha} (\mu)$. 
 There exists a constant $C>$ such that  $Mf(x) \leq C\widetilde{M } f(x)$ for all $x \in \mathbb{R}$, according to Lemma  \ref{lem3}.
It follows that 
\begin{eqnarray} \label{rel 8}
\Vert Mf \Vert_{(L^{1,+\infty}, L^{p})^{\alpha} (\mu)}
\leq C \Vert\widetilde{M } f\Vert_{(L^{1,+\infty}, L^{p})^{\alpha} (\mu)}.
\end{eqnarray}

Let $y \in \mathbb{R}$ and $r >0$. 
\begin{equation} \label{rel 1}
\Vert  (\widetilde{M } f ) \tau_{-y}\chi_{B_r} \Vert_{L^{1,+\infty}(\mu)} \leq \Vert  (\widetilde{M } f ) \chi_{B(y,r)} \Vert_{L^{1,+\infty}(\mu) }
\end{equation}
since $\tau_{-y}\chi_{B_r} \leq 1$ and the support  of  $\tau_{-y}\chi_{B_r} $ is a subset of  $B(y,r)$. 

Fix $\lambda >0$. We have  
\begin{eqnarray*}
\mu \left( \left\{
x \in B(y,r): \quad \widetilde{M } f (x)  >  \lambda
\right\} \right) &\leq& \mu \left( \left\{
x \in I(-y,3r):\quad \widetilde{M } f (x)  >  \frac{\lambda}{2}
\right\} \right) \nonumber \\
&+&  \mu \left( \left\{
x \in I(y,3r): \quad \widetilde{M } f (x)  >  \frac{\lambda}{2}
\right\}\right),
\end{eqnarray*}
thanks to Lemma \ref{lem 1}. 
But, as we can see in the proof of  \cite[Theorem 2.35]{De}, 
\begin{equation*}
\mu \left( \left\{
x \in I(-y,3r):\; \widetilde{M } f (x)  >  \frac{\lambda}{2}
\right\} \right) \lsim \frac{1}{\lambda}  \int_\mathbb{R} \vert f \vert (x) \widetilde{M }\chi_{I(-y,3r)} (x) d\mu (x)  
\end{equation*}

and 
\begin{equation*}
	\mu \left( \left\{
		x \in I(y,3r):\quad \widetilde{M } f (x)  >  \frac{\lambda}{2}
	\right\} \right) \lsim   \frac{1}{\lambda} \int_\mathbb{R} \vert f \vert (x) \widetilde{M }\chi_{I(y,3r)} (x) d\mu (x).  
	\end{equation*}
It follows that 
\begin{eqnarray*}
\lambda \mu\left( \left\{
	x \in B(y,r): \quad \widetilde{M } f (x)  >  \lambda
\right\} \right) &\lsim &  \int_\mathbb{R} \vert f \vert (x) \widetilde{M }  \chi_{I(-y,3r)} (x) d\mu (x)\\
&+&  \int_\mathbb{R} \vert f \vert (x) \widetilde{M }\chi_{I(y,3r)} (x) d\mu (x), 
\end{eqnarray*}
so that taking the supremum over all  $\lambda>0$, leads to 
\begin{eqnarray*}
\Vert  (\widetilde{M } f ) \chi_{B(y,r)} \Vert_{L^{1,+\infty}(\mu) } &\lsim&  \int_\mathbb{R} \vert f \vert (x)\widetilde{M } \chi_{I(-y,3r)} (x) d\mu (x)\\
& +&  \int_\mathbb{R} \vert f \vert (x) \widetilde{M }\chi_{I(y,3r)} (x) d\mu (x). 
\end{eqnarray*}
Taking this Estimate in  Relation (\ref{rel 1}) yields 
\begin{equation} \label{rel 4}
\begin{array}{lll}
\Vert  (\widetilde{M } f ) \tau_{-y}\chi_{B_r} \Vert_{L^{1,+\infty}(\mu)} &\lsim&  \int_\mathbb{R} \vert f \vert (x) \widetilde{M } \chi_{I(-y,3r)} (x) d\mu (x)\\
&+&  \int_\mathbb{R} \vert f \vert (x) \widetilde{M }\chi_{I(y,3r)} (x) d\mu (x). 
\end{array}
\end{equation}
Using the same argument as in the proof of Theorem \ref{maxi}, we have 
$$\begin{aligned}& \int_\mathbb{R} \vert f \vert (x) \widetilde{M } \chi_{I(-y,3r)} (x) d\mu (x) \\
&\qquad\qquad\qquad\qquad\lsim\int_{I(0,2r)} \tau_{y} \left|f(x)\right|  d\mu (x)  + \sum^{\infty}_{i=1} \frac{\mu(I(0,r))}{\mu(I(0,2^i  r))} \int_{I(0,2^{i+1} r) } \tau_{y} \left|f(x)\right|  d\mu(x), \end{aligned}$$
and 
$$\begin{aligned}
& \int_\mathbb{R} \vert f \vert (x) \widetilde{M } \chi_{I(y,3r)} (x) d\mu (x) \\
&\qquad\qquad\qquad\qquad\lsim \int_{I(0,2r)} \tau_{-y} \left|f(x)\right|  d\mu (x)  + \sum^{\infty}_{i=1} \frac{\mu(I(0,r))}{\mu(I(0,2^i  r))} \int_{I(0,2^{i+1} r) } \tau_{-y} \left|f(x)\right|  d\mu(x). \end{aligned}$$
Taking the above estimations in Relation (\ref{rel 4}), we have   
$$\begin{aligned}
&\Vert  (\widetilde{M } f ) \tau_{-y}\chi_{B_r} \Vert_{L^{1,+\infty}(\mu) } \\
&\qquad\qquad\qquad\qquad\lsim \left( \int_{I(0,2r)} \tau_{y} \left|f(x)\right|  d\mu (x)  + \sum^{\infty}_{i=1} \frac{\mu(I(0,r))}{\mu(I(0,2^i  r))} \int_{I(0,2^{i+1} r) } \tau_{y} \left|f(x)\right|  d\mu(x)  \right) \\
 &\qquad\qquad\qquad\qquad+  \left( \int_{I(0,2r)} \tau_{-y} \left|f(x)\right|  d\mu (x)  + \sum^{\infty}_{i=1} \frac{\mu(I(0,r))}{\mu(I(0,2^i  r))} \int_{I(0,2^{i+1} r) } \tau_{-y} \left|f(x)\right|  d\mu(x) \right).  
\end{aligned}$$
We end the proof as we did in Theorem \ref{maxi}.
\end{proof}

\end{document}